\documentclass[11pt]{amsart}
\usepackage{amsfonts}
\usepackage{amssymb,latexsym}
\usepackage{amsmath}
\usepackage{amsthm}
\usepackage{amssymb}
\usepackage{enumerate}
\newtheorem{theorem}{Theorem}[section]
\newtheorem{lemma}[theorem]{Lemma}

\newtheorem{corollary}[theorem]{Corollary}
\theoremstyle{definition}
\newtheorem{definition}[theorem]{Definition}

\newtheorem{example}[theorem]{Example}

\theoremstyle{remark}
\newtheorem{remark}[theorem]{Remark}
\numberwithin{equation}{section}

\def\dsum{\displaystyle\sum}
\def\dmax{\displaystyle\max}
\newcommand{\pa}{\partial}
\newcommand{\al}{\alpha}
\newcommand{\la}{\lambda}
\newcommand{\abc}[1]{\left( #1 \right)}%
\newcommand{\abz}[1]{\left[ #1 \right]}

\begin{document}
\title[Pogorelov type $C^2$ estimates for Sum Hessian equations]{Pogorelov type $C^2$ estimates for Sum Hessian equations and a rigidity theorem}

\author{Yue Liu}
\address{Yue Liu, School of Mathematical Science, Jilin University, Changchun, 130012, Jilin Province, P.R. China}\email{yue\_liu19@mails.jlu.edu.cn}

\author{Changyu Ren}
\address{Changyu Ren, School of Mathematical Science, Jilin University, Changchun, 130012, Jilin Province, P.R. China}\email{rency@jlu.edu.cn}

\thanks{Research of the last author is supported  by an NSFC Grant No. 11871243}

\begin{abstract}
We mainly study Pogorelov type $C^2$ estimates for solutions to the
Dirichlet problem of Sum Hessian equations. We establish
respectively Pogorelov type $C^2$ estimates for $k$-convex solutions
and admissible solutions under some conditions. Furthermore, we
apply such estimates to obtain a rigidity theorem for $k$-convex
solutions of Sum Hessian equations in Euclidean space.
\end{abstract}



\maketitle

\section{Introduction}

In this paper, we consider the following Dirichlet problem of Sum
Hessian equations,
\begin{equation}\label{1.1}
\left\{\begin{array}{lll}
 {{\sigma }_{k}}(D^2u)+\alpha{{\sigma
}_{k-1}}(D^2u)=f (x,u,Du),  &in &\Omega,\\
u=0,  & on&\pa{\Omega}.
\end{array} \right.
\end{equation}
Here, $u$ is an unknown function defined on $\Omega$, $\alpha$ is a
positive constant. Denote by $Du$ and $D^2u$ the gradient and the
Hessian of $u$. We also require $f>0$ and smooth enough with respect
to every variables. Let $\sigma_k(D^2u)=\sigma_k(\lambda(D^2u))$
denotes the $k$-th elementary symmetric function of the eigenvalues
of the Hessian matrix $D^2u$.~Namely, for
$\lambda=(\lambda_1,\cdots,\lambda_n)\in \mathbb{R}^n$,
$$
\sigma_k(\lambda)=\dsum_{1\leq i_1<\cdots<i_k\leq
n}\lambda_{i_1}\cdots\lambda_{i_k}.
$$

As we all know, the following classic $k$-Hessian equation
\begin{equation}\label{1.2}
\sigma_k(D^2u)=f(x,u,Du),~~~~~~~~x\in\Omega
\end{equation}
is an important research content in the field of fully nonlinear
partial differential equations and geometric analysis, which is
closely related to many important geometric problems, such as
\cite{BK,CNS3,CNS4,CY,GLL,GLM,O,P3,TW}. The equation \eqref{1.1} is
a natural extension of \eqref{1.2}, which is a Hessian type equation
formed by linear combination of the $k$ Hessian operators. These
Hessian type equations have been studied and applied in geometry.
Harvey and Lawson \cite{HL80} considered the special Lagrangian
equations which is one of this category.  Krylov \cite{Kry} and Dong
\cite{Dong} also considered equations close to this type and
obtained the curvature estimates using the concavity of the
operators. In \cite{GZ}, Guan and Zhang studied the curvature
estimates for equations of such type with the right hand side not
depending on gradient term but with coefficients depending on the
position of the hypersurfaces. The geometrical problems in the
hyperbolic space also reduce to equations of this type \cite{EGM}.

\par
An important problem in the study of Hessian equations (\ref{1.2})
is how to obtain $C^2$ estimates. For the relative work, we refer
the readers to \cite{CNS1,CNS3,CW,B,GLL,GLM,GRW}. Especially when
the right hand side function depends on the gradient, how to obtain
$C^2$ estimates for equations (\ref{1.2}) is a longstanding problem
\cite{GLL}.

\par
In this paper, we are interested in Pogorelov type $C^2$ estimates
for the Dirichlet problem \eqref{1.1}. Pogorelov type $C^2$
estimates is a type of interior $C^2$ estimates with boundary
information. This type of $C^2$ estimates were established at first
by Pogorelov \cite{P3} for ~Monge-Amp\`ere~ equations, which play an
important role in studying the regularity for fully nonlinear
equations. For example, Pogorelov type $C^2$ estimates were
established in order to study the regularity for degenerate
~Monge-Amp\`ere~ equations \cite{Blocki,Savin}. For $k$-Hessian
equations \eqref{1.2}, if the right hand side function is
$f=f(x,u)$, Chou and Wang ~\cite{CW} established Pogorelov $C^2$
estimates of $k$ convex solutions:
\begin{equation*}
(-u)^{1+\delta}\Delta u\leq C.
\end{equation*}
Sheng-Ubas-Wang \cite{SUW} established Pogorelov type curvature
estimates for the curvature equations of the graphic hypersurface
including the Hessian equations. If the right hand side function
depends on the gradient, Li-Wang and the last author~\cite{LRW1}
established Pogorelov type $C^2$ estimates for $k+1$ convex
solutions for equation (\ref{1.2}):
\begin{equation*}
(-u)\Delta u\leq C.
\end{equation*}
In particular, for $2$ convex solutions of the case $k=2$, the
following Pogorelov type $C^2$ estimates hold:
$$(-u)^\beta \Delta u\leq C.$$
Here, positive constants $\beta$ and $C$ depend on the domain
$\Omega$, $f$, $\displaystyle\sup_\Omega |u|$ and
$\displaystyle\sup_\Omega |Du|$.

\par
A natural question is whether Pogorelov type $C^2$ estimates are
still valid for the Dirichlet problem \eqref{1.1}?

\par
First, we give the definition of $k$ convex solution \cite{CNS3}:
\begin{definition}\label{k-convex}
For a domain $\Omega\subset \mathbb R^n$, a function $v\in
C^2(\Omega)$ is called $k$ convex if the eigenvalues $\lambda
(x)=(\lambda_1(x), \cdots, \lambda_n(x))$ of the Hessian $\nabla^2
v(x)$ is in $\Gamma_k$ for all $x\in \Omega$, where $\Gamma_k$ is
the Garding's cone
\[\Gamma_k=\{\lambda \in \mathbb R^n \ | \quad \sigma_m(\lambda)>0, \quad  m=1,\cdots,k\}.\]
\end{definition}
\par
The main results are as follows:

\par
\begin{theorem}\label{th1.1}
Let $\Omega$ be a bounded domain in $\mathbb{R}^n$, and let $u\in
C^4(\Omega)\cap C^{0,1}(\overline{\Omega})$ be a $k-1$ convex
solution for the Dirichlet problem (\ref{1.1}), where $f(x,u,p)\in
C^{1,1}(\overline{\Omega}\times \mathbb{R}\times \mathbb{R}^n)$ is a
positive function. Then we have
\par
(a) For $k=2,k=3$, there are some positive constants $\beta$ and
$C$, such that
\begin{equation}\label{1.3}
(-u)^\beta\Delta u\leq C.
\end{equation}
\par
(b) For $1<k\leq n$, if $f^{\frac{1}{k}}(x,u,p)$ is a convex
function with respect to $p$, there are some positive constants
$\delta$ and $C$, such that
\begin{equation}\label{1.4}
(-u)^{1+\delta}\Delta u\leq C.
\end{equation}
Here, $\delta>0$ can be arbitrarily small, the positive constants
$\beta$ and $C$ depend on the domain $\Omega$, $\alpha$, $f$,
$\displaystyle\sup_\Omega |u|$ and $\displaystyle\sup_\Omega |Du|$.
\end{theorem}
\begin{remark}
In \cite{LRW2}, Li-Wang and the last author proved that Sum Hessian
operator $\sigma_k+\al\sigma_{k-1}$ is elliptic operator in the set
$$\tilde\Gamma_k=\Gamma_{k-1}\cap\{\la|\sigma_k(\la)+\al\sigma_{k-1}(\la)>0\},$$
i.e., $\tilde\Gamma_k$ is the admissible set of the equations
(\ref{1.1}). Here, in addition to the requirements for $k-1$ convex
solutions, the right hand side function $f>0$ ensures the
ellipticity of the equations (\ref{1.1}).
\end{remark}

\begin{remark}
Similar to the result in \cite{LRW2}, because $f$ depends on the
gradient, the constant $\beta$ is large in theorem \ref{th1.1}, and
we cannot improve $\beta$ to 1 or $1+\delta$.
\end{remark}

\par
\begin{theorem}\label{th1.2}
Let $\Omega$ be a bounded domain in $\mathbb{R}^n$, and let $u\in
C^4(\Omega)\cap C^{0,1}(\overline{\Omega})$ be a $k$ convex solution
for the Dirichlet problem (\ref{1.1}), where $f(x,u,p)\in
C^{1,1}(\overline{\Omega}\times \mathbb{R}\times \mathbb{R}^n)$ is a
positive function. Then we have
\begin{equation*}
(-u)\Delta u\leq C.
\end{equation*}
Here, positive constant $C$ depends on the domain $\Omega$,
$\alpha$, $f$, $\displaystyle\sup_\Omega |u|$ and
$\displaystyle\sup_\Omega |Du|$.
\end{theorem}

\par
An application of the interior $C^2$ estimates is to prove rigidity
theorems for equations. For the following $k$-Hessian equations in
$n$-dimensional Euclidean spaces,
\begin{equation}\label{1.5}
\sigma_k(D^2u)=1,
\end{equation}
Chang and Yuan~\cite{CY} proposed a problem that: Are the entire
solutions of (\ref{1.5}) with lower bound only quadratic
polynomials?

\par
For $k=1$, equation (\ref{1.5}) is a linear equation. It is a
obvious result coming from the Liouville property of the harmonic
functions. For $k=n$, equation (\ref{1.5}) is ~Monge-Amp\`ere~
equation, which is a well known theorem. The work of
J\"{o}rgens~\cite{K}, Calabi~\cite {E}, Pogorelov~\cite{P1,P3}
proved that every entire strictly convex solution is a quadratic
polynomial. Then, Cheng and Yau \cite{CYa} gave another more
geometry proof.
 In 2003, Caffarelli and Li
~\cite{CL} extended the theorem of J\"{o}rgens, Calabi and
Pogorelov.

\par
For general $1<k<n$, there are few related results. For $k=2$, Chang
 and Yuan~\cite{CY} have proved that, if
$$D^2u\geq\delta-\sqrt{\dfrac{2n}{n-1}},$$
for any $\delta>0$, then the entire solution of the equation
(\ref{1.5}) only has quadratic polynomials. For general $k$,
Bao-Chen-Guan-Ji~\cite{BCGM} proved that strictly entire convex
solutions of equation (\ref{1.5}) satisfying a quadratic growth are
quadratic polynomials. Here, the quadratic growth means that: there
are some positive constants $c, b$ and sufficiently large $R$, such
that
\begin{equation}\label{1.6}
 u(x)\geq c|x|^2-b,  |x|\geq R.
\end{equation}
\cite{LRW1} relaxed the condition of strictly convex solutions in
\cite{BCGM} to $k+1$ convex solutions, and proved that the entire
$k+1$ convex solutions of the equations (\ref{1.5}) with quadratic
growth are quadratic polynomials. In 2016, Warren \cite{Warren}
proved that equation (\ref{1.5}) has non-polynomial entire $k$
convex solutions for $n\geq2k-1$ by constructing examples.

In this paper, we apply Pogorelov type $C^2$ estimates to prove the
rigidity theorem of the following Sum Hessian equations in Euclidean
space,
\begin{equation}\label{1.7}
\sigma_k(D^2u)+\al\sigma_{k-1}(D^2u)=1.
\end{equation}

\par
\begin{theorem}\label{th1.3}
The entire $k$ convex solutions of the equations (\ref{1.7}) defined
in $\mathbb{R}^n$ with quadratic growth (\ref{1.6}) are quadratic
polynomials.
\end{theorem}

At last, using the idea of  Warren \cite{Warren}, we can give an
example of non-polynomial entire $k-1$ convex solution.

\begin{example} When $n=3$, the
$1$-convex function
$$
u(x,y,t)=\dfrac{e^{4t}-1}{4}(x^2+y^2)+\dfrac{1}{16}(\dfrac{7e^{-4t}}{4}-\dfrac{e^{4t}}{4}-4t^2)
$$
solves
$$
\sigma_2(D^2u)+\sigma_1(D^2u)=1.
$$
\end{example}
\par
The organization of our paper is as follows: In section $2$, we give
the preliminary knowledge. The proofs of Theorem \ref{th1.1} and
Theorem \ref{th1.2} are respectively given in section 3 and section
4. In the last section, we prove the rigidity theorem of the
equations (\ref{1.7}).

\section{Preliminary}
\par
In this section, we give the preliminary knowledge related to our
theorems and their proofs. To make it convenient, we denote
\begin{equation*}
 S_k(\la):=\sigma_k(\la)+\alpha\sigma_{k-1}(\la).
\end{equation*}
For the related notation of $\sigma_k$, we use the denotation in
\cite{LRW2,RW1}. Let $\lambda=(\lambda_1,\cdots,\lambda_n)\in
\mathbb{R}^n$, we have
\par
(i) $S_k^{pp}(\lambda):=\dfrac{\partial
S_k(\lambda)}{\partial\lambda_p}=\sigma_{k-1}(\la|p)+\al\sigma_{k-2}(\la|p)=S_{k-1}(\lambda|p)$,
\quad $p=1,2,\cdots,n$;
\par
(ii)
$S_k^{pp,qq}(\lambda):=\dfrac{\partial^2S_k(\lambda)}{\partial\lambda_p\partial\lambda_q}=S_{k-2}(\lambda|pq)$,
\quad $p,q=1,2,\cdots,n$, \quad and $S_k^{pp,pp}(\la)=0$;
\par
(iii) $S_k(\lambda)=\lambda_iS_{k-1}(\lambda|i)+S_k(\lambda|i),$
$i=1,\cdots,n$;
\par
(iv)
$\dsum_{i=1}^nS_k(\lambda|i)=(n-k)S_k(\lambda)+\alpha\sigma_{k-1}(\lambda);$
\par
(v)
$\dsum_{i=1}^n\lambda_iS_{k-1}(\lambda|i)=kS_k(\lambda)-\alpha\sigma_{k-1}(\lambda).$

\par
In the joint work with Li, Wang and the last author \cite{LRW2}
proved that the admissible solution set of Sum Hessian operator
$S_k(\la)$ is
$$
\tilde\Gamma_k=\Gamma_{k-1}\cap\{\la|S_k>0\},
$$
and $S_k^{\frac{1}{k}}(\la)$ and
$(\dfrac{S_k}{S_l})^{\frac{1}{k-l}}$ are concave functions on
$\tilde\Gamma_k$. Using the method in \cite{GRW}, we can get the
following lemma.

\begin{lemma}\label{lem2.1}
Assume that $k>l$, for $\vartheta=\dfrac{1}{k-l}$, then for
$\la\in\tilde\Gamma_k$, we have
\begin{align}
-&\frac{S_k^{pp,qq}}{S_k}u_{pph}u_{qqh}+\dfrac{S_l^{pp,qq}}{S_l}u_{pph}u_{qqh}\\
&\geq \abc{\dfrac{(S_k)_h}{S_k}-\dfrac{(S_l)_h}{S_l}}
\abc{(\vartheta-1)\dfrac{(S_k)_h}{S_k}-(\vartheta+1)\dfrac{(S_l)_h}{S_l}}.\nonumber
\end{align}
Furthermore, for sufficiently small $\delta>0$, we have
\begin{align}\label{2.2}
-&S_k^{pp,qq}u_{pph}u_{qqh} +(1-\vartheta+\dfrac{\vartheta}{\delta})\dfrac{(S_k)_h^2}{S_k}\\
&\geq S_k(\vartheta+1-\delta\vartheta) \abz{\dfrac{(S_l)_h}{S_l}}^2
-\dfrac{S_k}{S_l}S_l^{pp,qq}u_{pph}u_{qqh}.\nonumber
\end{align}
\end{lemma}

The following Lemma comes from ~\cite{Ball}.

\begin{lemma} \label{lem2.2}
Denote by $Sym(n)$ the set of all $n\times n$ symmetric matrices.
Let $F$ be a $C^2$ symmetric function defined in some open subset
$\Psi \subset Sym(n)$. At any diagonal matrix $A\in\Psi$
 with distinct eigenvalues, let$\ddot{F}(B,B)$ be the second derivative of $C^2$ symmetric function $F$ in direction $B \in
Sym(n)$, then
\begin{align*}
 \ddot{F}(B,B) =  \sum_{j,k=1}^n {\ddot{f}}^{jk}
B_{jj}B_{kk} + 2 \sum_{j < k} \frac{\dot{f}^j -
\dot{f}^k}{{\kappa}_j - {\kappa}_k} B_{jk}^2.
\end{align*}
\end{lemma}
\par

\begin{lemma}\label{lem2.3a}
Assume $\lambda=(\lambda_1,\cdots,\lambda_n)\in\tilde\Gamma_k$. If
$S_{k+1}(\la)>0$, then $\la\in\tilde\Gamma_{k+1}$.
\end{lemma}
\par
\begin{proof}
By the definition of $\tilde\Gamma_{k+1}$, we only need to prove
$\sigma_k(\la)>0$. If $\sigma_k(\la)\leq 0$, since $S_k(\la)>0,
S_{k+1}(\la)>0$, then
$$
\al\sigma_{k-1}>|\sigma_k|,\quad \sigma_{k+1}>|\al\sigma_k|,
$$
namely, $\al\sigma_{k-1}\sigma_{k+1}>\al\sigma_k^2$, this
contradicts Newton's inequality.
\end{proof}

\begin{remark}
Using Lemma \ref{lem2.3a}, we can define $\tilde\Gamma_k$ as
$$
\tilde\Gamma_k=\{\la\in \mathbb{R}^n| S_m(\la)>0,\quad
m=1,2,\cdots,k\}.
$$
\end{remark}

\begin{lemma}\label{lem2.4a}
Assume that $\lambda=(\lambda_1,\cdots,\lambda_n)\in\tilde\Gamma_k$,
$1\leq k\leq n$, $\lambda_1\geq\cdots\geq\lambda_n$, then
\par
(a) For any $1\leq s<k$, we have
\begin{equation}\label{2.54}
S_s(\lambda)>\lambda_1\cdots\lambda_s+\al
\lambda_1\cdots\lambda_{s-1};
\end{equation}
\par
(b) For any $1\leq j\leq k-1$, there exists a positive constant
$\theta$ depending on $n,k$, such that
\begin{equation}\label{2.53}
S_k^{jj}(\la)\geq \dfrac{\theta S_k(\la)}{\la_j}.
\end{equation}
\end{lemma}
\begin{proof}
(a) Using $\la\in\tilde\Gamma_k\subset\Gamma_s$, we have
$$
\la_1>0, ~\la_2>0, \cdots, ~\la_s>0,
$$
and
$$
S_s(\la|1)>0, ~S_{s-1}(\la|12)>0, \cdots, ~S_{1}(\la|12\cdots s)>0.
$$
Using the above inequalities, we get
\begin{align*}
S_s(\la)=&\la_1S_{s-1}(\la|1)+S_s(\la|1)\\
>&\la_1S_{s-1}(\la|1)=\la_1\la_2S_{s-2}(\la|12)+\la_1S_{s-1}(\la|12)\\
>&\la_1\la_2S_{s-2}(\la|12)=\cdots\\
>&\la_1\la_2\cdots\la_{s-1}(\la_s+\al).
\end{align*}
\par
(b) If $S_k(\la|j)\leq 0$, obviously we have
$$
S_k^{jj}(\la)=\dfrac{S_k(\la)-S_k(\la|j)}{\la_j}\geq\dfrac{S_k(\la)}{\la_j}.
$$
If $S_k(\la|j)> 0$, then $(\la|j)\in\tilde\Gamma_k$ by Lemma
\ref{lem2.3a}. Using \eqref{2.54}, we obtain
$$
S_k^{jj}(\la)>\dfrac{\la_1\cdots\la_k+\al\la_1\cdots\la_{k-1}}{\la_j}.
$$
On the other hand, $(\la|j)\in\tilde\Gamma_k$ implies $\la_k>0$ and
$$
\la_k+\cdots+\la_n+\al>0.
$$
So there exists some constant $\theta$ depending on $n,k$, such that
$$
\theta S_k(\la)\leq \la_1\cdots\la_k+\al\la_1\cdots\la_{k-1}.
$$
Then we obtain \eqref{2.53}.
\end{proof}

\begin{lemma} \label{lem2.5}
For $\lambda\in\tilde\Gamma_k$, we have
$S_k^2(\la)-S_{k-1}(\la)S_{k+1}(\la)\geq0.$
\end{lemma}
\begin{proof}
Since $\lambda\in\tilde\Gamma_k$, $S_k(\la)>0, S_{k-1}(\la)>0$. If
$S_{k+1}(\la)\leq 0$, the Lemma obviously holds.
\par
If $S_{k+1}(\la)> 0$, Lemma \ref{lem2.3a} implies
$\la\in\tilde\Gamma_{k+1}\subset\Gamma_k$. Thus
\begin{align*}
 S_k^2-S_{k-1}S_{k+1}
       =&(\sigma_k+\al\sigma_{k-1})^2-(\sigma_{k-1}+\al\sigma_{k-2})(\sigma_{k+1}+\al\sigma_{k})\\
       =&\sigma_k^2+\alpha^2\sigma_{k-1}^2+\alpha\sigma_k\sigma_{k-1}-(\sigma_{k-1}\sigma_{k+1}+\alpha^2\sigma_{k-2}\sigma_{k}+\alpha\sigma_{k-2}\sigma_{k+1})\\
       \geq&0.
\end{align*}
In the above inequality, we used Newton's inequality and the
following inequality which obtained from Newton's inequality
\begin{align*}
\sigma_k\sigma_{k-1}\geq&\sigma_{k-2}\sigma_{k+1}.
\end{align*}
\end{proof}

\begin{lemma}\label{lem2.4}
Assume that $\lambda=(\lambda_1,\cdots,\lambda_n)\in\Gamma_k$,
$1\leq k\leq n$, $\lambda_1\geq\cdots\geq\lambda_n$, and $S_k\leq
N_0$ for some positive constant $N_0>0$. Then
\par
(a) There exists some positive constant $K_0$, such that
\begin{equation}\label{2.4}
\lambda_{k-1}\leq(\dfrac{N_0}{\al})^{\frac{1}{k-1}},~~~~~~\lambda_n>-K_0;
\end{equation}
\par
(b) Denote $\kappa_i=\lambda_i+K_0$. For any $1\leq i\leq n$, we
have
\begin{equation}\label{2.5}
2\kappa_1^{k+2}S_k^{11}(\la)\geq\kappa_i^{k+2}S_k^{ii}(\la),
\end{equation}
if $\lambda_1$ is sufficiently large;
\par
(c) For any $\varepsilon_0>0$, we have
\begin{equation}\label{2.51}
S_k(\la)>(1-\varepsilon_0)\lambda_1S_k^{11}(\la),
\end{equation}
if $\lambda_1$ is sufficiently large;
\par
(d) For any $1\leq i\leq n$, there exists some positive constant
$C_0$, such that
\begin{equation}\label{2.52}
C_0S_k(\la)\geq\lambda_iS_k^{ii}(\la).
\end{equation}
\end{lemma}

\begin{proof}
(a) Since $\lambda\in\Gamma_k$, by Lemma 11 in \cite{RW2}, we have
\begin{equation*}
\begin{aligned}
 N_0\geq S_k=\sigma_k+\alpha\sigma_{k-1}>&\alpha\sigma_{k-1}
    \geq\alpha\lambda_1\cdots\lambda_{k-1}
    >\alpha\lambda_{k-1}^{k-1}.
\end{aligned}
\end{equation*}
Thus, $\lambda_{k-1}\leq (\dfrac{N_0}{\al})^{\frac{1}{k-1}}$. Using
$\lambda_k\geq0$, we have
$$\lambda_n>-(n-k)\lambda_k\geq -(n-k)\lambda_{k-1}>-K_0.$$

\par
(b) From \eqref{2.4}, we get $\kappa_i>0$. We divide into two cases
to prove \eqref{2.5}.
\par
\textbf{Case 1:} $\lambda_i\geq\lambda_1^{1/(k+2)}$. Since
\begin{align}\label{2.6}
2\kappa_1^2S_k^{11}-\kappa_i^2S_k^{ii}=&\kappa_1^2S_k^{11}+\kappa_1^2[\lambda_iS_k^{11,ii}+S_{k-1}(\lambda|1i)]\\
&-\kappa_i^2[\lambda_1S_k^{11,ii}+S_{k-1}(\lambda|1i)]\nonumber\\
    \geq&\kappa_1^2S_k^{11}+(\kappa_1^2-\kappa_i^2)S_{k-1}(\lambda|1i)\nonumber\\
    \geq&(\kappa_1^2-\kappa_i^2)[S_k^{11}+S_{k-1}(\lambda|1i)].
    \nonumber
\end{align}
Using $\lambda\in\Gamma_k$, we have
\begin{equation*}
S_{k-1}(\lambda|1i)=\sigma_{k-1}(\lambda|1i)+\alpha\sigma_{k-2}(\lambda|1i)\geq\sigma_{k-1}(\lambda|1i)>C_{n-2}^{k-1}\dfrac{\lambda_2\cdots\lambda_k\lambda_n}{\lambda_i},
\end{equation*}
and
\begin{equation}\label{e2.10}
S_k^{11}=\sigma_{k-1}(\lambda|1)+\alpha\sigma_{k-2}(\lambda|1)\geq\alpha\lambda_2\cdots\lambda_{k-1}.
\end{equation}
In view of \eqref{2.4}, we have $|\lambda_k\lambda_n|<K_0^2$. Using
the above two formulas, we have
\begin{equation}\label{2.7}
S_k^{11}+S_{k-1}(\lambda|1i)>\lambda_2\cdots\lambda_{k-1}(\alpha+C_{n-2}^{k-1}\dfrac{\lambda_k\lambda_n}{\lambda_i})>0,
\end{equation}
if $\lambda_1$ is sufficiently large. Combining \eqref{2.7} with
\eqref{2.6}, we can get \eqref{2.5}.
\par
\textbf{Case 2:} $\lambda_i<\lambda_1^{1/(k+2)}$. In this case,
$2\kappa_1>\kappa_i^{k+2}$, using \eqref{2.53}, we have
\begin{align*}
\kappa_1^{k+1}S_k^{11}>\kappa_1^k\lambda_1S_k^{11}\geq\kappa_1^kc_0S_k\geq\kappa_1c_0S_k\lambda_1^{k-1}\geq
S_k^{ii},
\end{align*}
when $\lambda_1$ is sufficiently large. Then we get
$$2\kappa_1^{k+2}S_k^{11}\geq\kappa_i^{k+2}S_k^{ii}.$$
\par
(c) Using \eqref{e2.10} and
$$\sigma_k(\lambda|1)>C_{n-1}^k\lambda_2\cdots\lambda_k\lambda_n,$$
we get
\begin{align*}
  S_k-(1-\varepsilon_0)\lambda_1S_k^{11}
    =&\varepsilon_0\lambda_1S_k^{11}+S_k(\lambda|1)\\
    >&\varepsilon_0\lambda_1S_k^{11}+\sigma_k(\lambda|1)\\
    \geq&\lambda_2\cdots\lambda_{k-1}(\varepsilon_0\lambda_1\alpha+C_{n-1}^k\lambda_k\lambda_n).
\end{align*}
By \eqref{2.4}, we know $|\lambda_k\lambda_n|<K_0^2$. Thus when
$\lambda_1$ is sufficiently large, the above formula is
non-negative.
\par
(d) Since $\lambda\in\Gamma_k$, we have
$$S_k=\sigma_k+\alpha\sigma_{k-1}\geq\alpha\lambda_1\cdots\lambda_{k-1},$$
and
$$\sigma_k(\lambda|i)>C_n^k\lambda_1\cdots\lambda_{k-1}\lambda_n,$$
which implies
\begin{equation*}
\begin{aligned}
  C_0S_k-\lambda_iS_k^{ii}
    =&(C_0-1)S_k+S_k(\lambda|i)\\
    =&(C_0-1)S_k+\sigma_k(\lambda|i)+\alpha\sigma_{k-1}(\lambda|i)\\
    >&\lambda_1\cdots\lambda_{k-1}[(C_0-1)\alpha+C_n^k\lambda_n].
\end{aligned}
\end{equation*}
By \eqref{2.4}, when $C_0>1+\dfrac{K_0C_n^k}{\al}$, the above
formula is non-negative.
\end{proof}

\section{Pogorelov type $C^2$ estimates for $k-1$ convex solutions}
\par
In this section, we prove theorem \ref{th1.1}. First, for the case
of $k=2,3$, we establish Pogorelov type $C^2$ estimates for $k-1$
convex solutions of the Dirichlet problem \eqref{1.1}, and then we
give a proof for the case that $f^{\frac{1}{k}}(x,u,p)$ is convex
with respect to $p$.
\par
{\bf Proof of \eqref{1.3}:} Let $\lambda_1=\lambda_1(x)$ denote the
biggest eigenvalue of $D^2u$. Consider the test function:
$$\phi(x)=\lambda_1(-u)^\beta \exp\{\dfrac{\varepsilon}{2}|Du|^2+\dfrac{a}{2}|x|^2\},$$
where $\beta, \varepsilon,a$ are undetermined positive constants.
Using the trick of Brendle-Choi-Daskalopoulos~\cite{BCD,Yang},
assume test function $\phi(x)$ obtains its maximum at point $x_0$.
Define a function $\psi(x)$:
$$\psi(x)=\phi(x_0)(-u)^{-\beta} \exp\{-\dfrac{\varepsilon}{2}|Du|^2-\dfrac{a}{2}|x|^2\}.$$
Since
\begin{align*}
\phi(x_0)=&\psi(x)(-u)^{\beta}\exp\{\dfrac{\varepsilon}{2}|Du|^2+\dfrac{a}{2}|x|^2\}\\
\geq&\phi(x)= \lambda_1(-u)^\beta
\exp\{\dfrac{\varepsilon}{2}|Du|^2+\dfrac{a}{2}|x|^2\},
\end{align*}
we have $\psi(x)\geq\lambda_1(x)$ and $\psi(x_0)=\lambda_1(x_0)$. By
rotating the coordinate, we assume that $(u_{ij})$ is a diagonal
matrix at $x_0$, and
$u_{ii}=\lambda_i,~\lambda_1\geq\cdots\geq\lambda_n$.
\par
Let $\mu$ denote the multiplicity of the biggest eigenvalue at
$x_0$, namely,
$\lambda_1=\cdots=\lambda_\mu>\lambda_{\mu+1}\geq\cdots\geq\lambda_n$.
Then from Brendle-Choi-Daskalopoulos result(Lemma 5 in \cite{BCD}),
we have the followings exist at $x_0$:
\begin{equation*}
\psi=\la_1=u_{11},
\end{equation*}
\begin{equation*}
u_{kji}=\psi_{i}\delta_{kj},\quad 1\leq k,j\leq\mu,
\end{equation*}
\begin{equation*}
\psi_{ii}\geq
u_{11ii}+2\dsum_{j>\mu}\dfrac{1}{\la_1-\la_j}u_{1ji}^2,
\end{equation*}

Now we restrict all calculations at point $x_0$. Since the function
$$\tilde \phi=\psi(x)(-u)^\beta \exp\{\dfrac{\varepsilon}{2}|Du|^2+\dfrac{a}{2}|x|^2\}=\phi(x_0)$$
has constant value, at $x_0$, $(\ln\tilde\phi)_i=0,\quad
(\ln\tilde\phi)_{ii}=0$, using the above formulas of $\psi$, we
obtain

\begin{equation}\label{3.1}
\dfrac{\psi_i}{\psi}+\dfrac{\beta u_i}{u}+\varepsilon
u_iu_{ii}+ax_i=\dfrac{u_{11i}}{u_{11}}+\dfrac{\beta
u_i}{u}+\varepsilon u_iu_{ii}+ax_i=0,
\end{equation}
and
\begin{align*}
0=&\dfrac{\psi_{ii}}{\psi}-\dfrac{\psi_{i}^2}{\psi^2}+\dfrac{\beta
u_{ii}}{u}-\dfrac{\beta u_i^2}{u^2}+\displaystyle\sum_s
\varepsilon u_su_{sii}+\varepsilon u_{ii}^2+a\\
 \geq&\dfrac{\beta u_{ii}}{u}-\dfrac{\beta
u_i^2}{u^2}+\dfrac{u_{11ii}}{u_{11}}+2\displaystyle\sum_{j>\mu}\dfrac{u_{1ji}^2}{\lambda_1(\lambda_1-\lambda_j)}-\dfrac{u_{11i}^2}{u_{11}^2}+\displaystyle\sum_s
\varepsilon u_su_{sii}+\varepsilon u_{ii}^2+a.
\end{align*}
In the above inequality, contracting with $S_k^{ii}$, we have
\begin{align}\label{3.2}
0\geq&\dfrac{\beta S_k^{ii}u_{ii}}{u}-\dfrac{\beta
S_k^{ii}u_i^2}{u^2}+\dfrac{S_k^{ii}u_{11ii}}{u_{11}}+2\dsum_{j>\mu}\dfrac{S_k^{ii}u_{1ji}^2}{\lambda_1(\lambda_1-\lambda_j)}-\dfrac{S_k^{ii}u_{11i}^2}{u_{11}^2}\\
& +\dsum_s\varepsilon u_sS_k^{ii}u_{sii}+\varepsilon
S_k^{ii}u_{ii}^2+a\dsum_iS_k^{ii}.\nonumber
\end{align}
\par
At $x_0$, differentiating equation \eqref{1.1} twice, we have
\begin{equation}\label{3.3}
S_k^{ii}u_{iij}=f_j+f_uu_j+f_{p_j}u_{jj},
\end{equation}
and
\begin{equation}\label{3.4}
S_k^{ii}u_{iijj}+S_k^{pq,rs}u_{pqj}u_{rsj}\geq-C-Cu_{11}+f_{p_jp_j}u_{jj}^2+\dsum_sf_{p_s}u_{sjj}.
\end{equation}
Substituting \eqref{3.4} into \eqref{3.2}, we have
\begin{align}\label{3.5}
0\geq& \dfrac{\beta S_k^{ii}u_{ii}}{u}-\dfrac{\beta S_k^{ii}u_i^2}{u^2}+\dfrac{1}{u_{11}}[-C-Cu_{11}+f_{p_1p_1}u_{11}^2+\dsum_sf_{p_s}u_{s11}-K(S_k)_1^2\\
       &+K(S_k)_1^2-S_k^{pq,rs}u_{pq1}u_{rs1}]+2\dsum_{j>\mu}\dfrac{S_k^{11}u_{11j}^2}{\lambda_1(\lambda_1-\lambda_j)}-\dfrac{S_k^{ii}u_{11i}^2}{u_{11}^2}\nonumber\\
       &+\dsum_s\varepsilon u_sS_k^{ii}u_{sii}+\varepsilon S_k^{ii}u_{ii}^2+a\dsum_iS_k^{ii},\nonumber
\end{align}
where $K$ is a positive constant which will be determined later.
\par
Combining \eqref{3.1} with \eqref{3.3}, we obtain
\begin{equation*}
\dfrac{1}{u_{11}}\dsum_sf_{p_s}u_{s11}+\dsum_s\varepsilon
u_sS_k^{ii}u_{sii}\geq-\dsum_s\dfrac{\beta u_sf_{p_s}}{u}-C.
\end{equation*}
By Lemma \ref{lem2.2},
\begin{equation*}
\begin{aligned}
 -S_k^{pq,rs}u_{pq1}u_{rs1}& =-S_k^{pp,qq}u_{pp1}u_{qq1}+\dsum_{p\neq q}S_k^{pp,qq}u_{pq1}^2\\
                                & \geq-S_k^{pp,qq}u_{pp1}u_{qq1}+2\dsum_{j\neq1}S_k^{11,jj}u_{11j}^2.
\end{aligned}
\end{equation*}
Let $K=\dfrac{k-1}{kS_k}$. Using Lemma \ref{lem2.1}, we have
\begin{equation*}
K(S_k)_1^2-S_k^{pp,qq}u_{pp1}u_{qq1}\geq0.
\end{equation*}
Because $\lambda(D^2u)\in\Gamma_{k-1}\subset\Gamma_1$, that is
$\Delta u=\sigma_1(D^2u)>0$,  so we have $u<0$ at $x_0$. By (v) in
section 2,
$$
\dfrac{\beta
S_k^{ii}u_{ii}}{u}=\dfrac{\beta}{u}(kS_k-\al\sigma_{k-1})>\dfrac{\beta
kS_k}{u}.
$$

Using the above inequalities, \eqref{3.5} becomes
\begin{align}\label{3.6}
-\dfrac{C}{u}\geq& -\dfrac{\beta S_k^{ii}u_i^2}{u^2}+\dfrac{2}{u_{11}}\dsum_{j\neq 1}S_k^{11,jj}u_{11j}^2+2\dsum_{j>\mu}\dfrac{S_k^{11}u_{11j}^2}{\lambda_1(\lambda_1-\lambda_j)}\\
                  &-\dfrac{S_k^{ii}u_{11i}^2}{u_{11}^2}+\varepsilon S_k^{ii}u_{ii}^2+a\dsum_iS_k^{ii}+\Big(f_{p_1p_1}-\dfrac{(k-1)f_{p_1}^2}{kf}\Big)u_{11}-C.\nonumber
\end{align}

\par
Next, we divide into two cases to deal with \eqref{3.6}.
\par
{\bf Case A:} $\lambda_n\geq-\dfrac{\lambda_1}{2}$. For index
$i>\mu$, using \eqref{3.1}, we have
\begin{equation}\label{3.7}
-\dfrac{\beta u_i^2}{u^2}\geq -\dfrac{2u_{11i}^2}{\beta
u_{11}^2}-\dfrac{4}{\beta}(\varepsilon
u_iu_{ii})^2-\dfrac{4}{\beta}(ax_i)^2.
\end{equation}
For any $j>\mu$ and $\beta>6$, we have
\begin{align}\label{3.8}
\dfrac{2}{u_{11}}&S_k^{11,jj}u_{11j}^2+2\dfrac{S_k^{11}}{\lambda_1(\lambda_1-\lambda_j)}u_{11j}^2-\dfrac{2+\beta}{\beta}\dfrac{S_k^{jj}}{u_{11}^2}u_{11j}^2\\
=&\dfrac{(\beta-2)\lambda_1+(\beta+2)\lambda_j}{\beta(\lambda_1-\lambda_j)}\dfrac{S_k^{jj}u_{11j}^2}{\lambda_1^2}\geq0.\nonumber
\end{align}
Substituting \eqref{3.7}, \eqref{3.8} into \eqref{3.6}, we get
\begin{align}\label{3.9}
-\dfrac{C}{u}\geq& -\dsum_{i=1}^\mu\dfrac{\beta
S_k^{ii}u_i^2}{u^2}-\dsum_{i=1}^\mu\dfrac{S_k^{ii}u_{11i}^2}{u_{11}^2}+\varepsilon
S_k^{ii}u_{ii}^2-\dfrac{4}{\beta}\dsum_{i>\mu}S_k^{ii}(\varepsilon
u_iu_{ii})^2\\
&-\dfrac{4}{\beta}\dsum_{i>\mu}S_k^{ii}(ax_i)^2+a\dsum_iS_k^{ii}+\Big(f_{p_1p_1}-\dfrac{(k-1)f_{p_1}^2}{kf}\Big)u_{11}-C.\nonumber
\end{align}
For index $i\leq\mu$, using \eqref{3.1}, we have
\begin{equation}\label{3.10}
-\dfrac{u_{11i}^2}{u_{11}^2}\geq-2(\dfrac{\beta
u_i}{u})^2-4(\varepsilon u_iu_{ii})^2-4(ax_i)^2.
\end{equation}
Substituting \eqref{3.10} into \eqref{3.9}, we get
\begin{align}\label{3.11}
-\dfrac{C}{u}&+\dsum_{i=1}^\mu\dfrac{(\beta+2\beta^2)S_k^{ii}u_i^2}{u^2}\\
 &\geq\varepsilon S_k^{ii}u_{ii}^2-4\dsum_{i=1}^\mu S_k^{ii}(\varepsilon u_iu_{ii})^2-4\dsum_{i=1}^\mu S_k^{ii}(ax_i)^2-\dfrac{4}{\beta}\dsum_{i>\mu}S_k^{ii}(\varepsilon u_iu_{ii})^2\nonumber\\
      &-\dfrac{4}{\beta}\dsum_{i>\mu}S_k^{ii}(ax_i)^2+a\dsum_iS_k^{ii}+\Big(f_{p_1p_1}-\dfrac{(k-1)f_{p_1}^2}{kf}\Big)u_{11}-C.\nonumber
\end{align}
We choose $\varepsilon$, such that
$$\varepsilon>8\varepsilon^2\displaystyle\max_\Omega|Du|^2.$$
Note that $i\leq \mu$, $u_{ii}=u_{11}=\lambda_1$, thus for the above
selected $\varepsilon$,
\begin{equation}\label{3.12}
\dfrac{3}{4}\varepsilon S_k^{ii}u_{ii}^2-4\dsum_{i=1}^\mu
S_k^{ii}(\varepsilon u_iu_{ii})^2-4\dsum_{i=1}^\mu S_k^{ii}(ax_i)^2
      -\dfrac{4}{\beta}\dsum_{i>\mu}S_k^{ii}(\varepsilon
      u_iu_{ii})^2>0,
\end{equation}
when $\lambda_1$ is sufficiently large. By (iv) in section 2, for
$k=2,3$, when $\lambda\in\tilde\Gamma_k$,
\begin{equation}\label{3.13}
\begin{aligned}
&\dsum_{i=1}^nS_2^{ii}=(n-1)\sigma_1+n\al>(n-1)\sigma_1,\\
&\dsum_{i=1}^nS_3^{ii}=(n-2)\sigma_2+(n-1)\al\sigma_1>(n-1)\al\sigma_1.
\end{aligned}
\end{equation}
Note that, when $\lambda\in\tilde\Gamma_2$,
$$
S_2^{11}=\lambda_2+\cdots+\lambda_n+\alpha>0,
$$
which implies $\sigma_1>\lambda_1-\al$. Taking $a$ sufficiently
large, we have
\begin{equation}\label{3.14}
\dfrac{a}{3}\dsum_{i=1}^nS_k^{ii}+\Big(f_{p_1p_1}-\dfrac{(k-1)f_{p_1}^2}{kf}\Big)u_{11}-C>0.
\end{equation}
Choose $\beta>a^2$, such that
\begin{equation}\label{3.15}
-\dfrac{4}{\beta}\dsum_{i>\mu}S_k^{ii}(ax_i)^2+\dfrac{a}{3}\dsum_{i>\mu}S_k^{ii}>0.
\end{equation}
Substituting \eqref{3.12}, \eqref{3.14}, \eqref{3.15} into
\eqref{3.11}, we get
\begin{align*}
-\dfrac{C}{u}+\dsum_{i=1}^\mu\dfrac{(\beta+2\beta^2)S_k^{ii}u_i^2}{u^2}
\geq&\dfrac{\varepsilon}{4}\dsum_{i=1}^\mu
S_k^{ii}u_{ii}^2+\dfrac{a}{3}\dsum_iS_k^{ii}.
\end{align*}
When
\begin{align*}
-\dfrac{C}{u}\geq
\dsum_{i=1}^\mu\dfrac{(\beta+2\beta^2)S_k^{ii}u_i^2}{u^2},
\end{align*}
we get
\begin{align*}
-\dfrac{2C}{u}\geq&\dfrac{a}{3}\dsum_iS_k^{ii}.
\end{align*}
 Using  \eqref{3.13}, we obtain \eqref{1.3}.
 \par
When
\begin{align*}
-\dfrac{C}{u}
<\dsum_{i=1}^\mu\dfrac{(\beta+2\beta^2)S_k^{ii}u_i^2}{u^2},
\end{align*}
we get
\begin{align*}
2\dsum_{i=1}^\mu\dfrac{(\beta+2\beta^2)S_k^{ii}u_i^2}{u^2}
\geq&\dfrac{\varepsilon}{4}\dsum_{i=1}^\mu S_k^{ii}u_{ii}^2.
\end{align*}
Since $u_{11}=\cdots=u_{\mu\mu}$, then
$S_k^{11}=\cdots=S_k^{\mu\mu}$, and then we get
$$
(-u)^2u_{11}^2\leq C
$$
by the above inequality.

\vskip 2mm

\par
{\bf Case B:} $\lambda_n<-\dfrac{\lambda_1}{2}$. Using \eqref{3.6}
and \eqref{3.10}, we get
\begin{align}\label{3.16}
-\dfrac{C}{u}+\dfrac{(\beta+2\beta^2)S_k^{ii}u_i^2}{u^2}\geq&\varepsilon
S_k^{ii}u_{ii}^2-4 S_k^{ii}(\varepsilon u_iu_{ii})^2-4
S_k^{ii}(ax_i)^2\\
&+a\dsum_i
S_k^{ii}+\Big(f_{p_1p_1}-\dfrac{(k-1)f_{p_1}^2}{kf}\Big)u_{11}-C.\nonumber
\end{align}
Note that $S_k^{nn}\geq \cdots\geq S_k^{11}$ and
$\lambda_n^2>\dfrac{\lambda_1^2}{4}$. We choose
$\varepsilon>8\varepsilon^2\displaystyle\max_\Omega|Du|^2$, then
\begin{align}\label{3.17}
\dfrac{3}{4}\varepsilon S_k^{ii}u_{ii}^2-4 S_k^{ii}(\varepsilon
u_iu_{ii})^2-4 S_k^{ii}(ax_i)^2>0,
\end{align}
if $\lambda_1$ is sufficiently large. Let $a$ be sufficiently large
satisfying \eqref{3.14}. Substituting \eqref{3.14}, \eqref{3.17}
into \eqref{3.16}, we get
\begin{equation*}
-\dfrac{C}{u}+\dfrac{(\beta+2\beta^2)S_k^{ii}u_i^2}{u^2}\geq\dfrac{\varepsilon}{4}S_k^{ii}u_{ii}^2+\dfrac{2a}{3}\dsum_iS_k^{ii}.
\end{equation*}
Using again \eqref{3.13}, we obtain \eqref{1.3}.

\par
{\bf Proof of \eqref{1.4}:} Let $\lambda_1=\lambda_1(x)$ denote the
biggest eigenvalue of $D^2u$. We consider the following test
function:
$$\phi(x)=\lambda_1(-u)^\beta \exp\{\dfrac{\varepsilon}{2}|Du|^2\},$$
where $\beta=1+\delta'=1+2\delta$. We may assume that
$0<\delta<\dfrac{1}{3}$.
\par
Suppose that function $\phi$ achieves its maximum value in $\Omega$
at some point $x_0$, and similar to the process of obtaining
\eqref{3.1} and \eqref{3.6}, let $a=0$ and use the convexity of
$f^{\frac{1}{k}}(x,u,p)$, we obtain
\begin{equation}\label{3.18}
\dfrac{\beta u_i}{u}+\dfrac{u_{11i}}{u_{11}}+\varepsilon
u_iu_{ii}=0,
\end{equation}
\begin{align}\label{3.19}
-\dfrac{C}{u}\geq& -\dfrac{\beta S_k^{ii}u_i^2}{u^2}+\dfrac{2}{u_{11}}\dsum_{j\neq 1}S_k^{11,jj}u_{11j}^2+2\dsum_{j>\mu}\dfrac{S_k^{11}u_{11j}^2}{\lambda_1(\lambda_1-\lambda_j)}\\
                  &-\dfrac{S_k^{ii}u_{11i}^2}{u_{11}^2}+\varepsilon S_k^{ii}u_{ii}^2-C.\nonumber
\end{align}

\par
Next, we divide into two cases to deal with \eqref{3.19}.
\par
{\bf Case A:} $\lambda_n\geq-\dfrac{\delta\lambda_1}{3}$. For index
$i>\mu$, using \eqref{3.18}, we have
\begin{equation}\label{3.20}
-\dfrac{\beta u_i^2}{u^2}\geq -\dfrac{(1+\delta)u_{11i}^2}{\beta
u_{11}^2}-\dfrac{1+1/\delta}{\beta}(\varepsilon u_iu_{ii})^2.
\end{equation}
For any $j>\mu$, we have
\begin{align}\label{3.21}
\dfrac{2}{u_{11}}&S_k^{11,jj}u_{11j}^2+2\dfrac{S_k^{11}}{\lambda_1(\lambda_1-\lambda_j)}u_{11j}^2-\dfrac{1+\delta+\beta}{\beta}\dfrac{S_k^{jj}}{u_{11}^2}u_{11j}^2\\
&=\dfrac{\delta\lambda_1+(2+3\delta)\lambda_j}{\beta(\lambda_1-\lambda_j)}\dfrac{S_k^{jj}u_{11j}^2}{\lambda_1^2}\geq0.\nonumber
\end{align}
Substituting \eqref{3.20}, \eqref{3.21} into \eqref{3.19}, we obtain
\begin{align}\label{3.22}
-\dfrac{C}{u}\geq& -\dsum_{i=1}^\mu\dfrac{\beta
S_k^{ii}u_i^2}{u^2}-\dsum_{i=1}^\mu\dfrac{S_k^{ii}u_{11i}^2}{u_{11}^2}+\varepsilon
S_k^{ii}u_{ii}^2-\dfrac{1+1/\delta}{\beta}\dsum_{i>\mu}S_k^{ii}(\varepsilon
u_iu_{ii})^2-C.
\end{align}
For index $i\leq\mu$, using \eqref{3.18}, we have
\begin{equation}\label{3.23}
-\dfrac{u_{11i}^2}{u_{11}^2}\geq-2(\dfrac{\beta
u_i}{u})^2-2(\varepsilon u_iu_{ii})^2.
\end{equation}
Substituting \eqref{3.23} into \eqref{3.22}, we obtain
\begin{align}\label{3.24}
-&\dfrac{C}{u}+\dsum_{i=1}^\mu\dfrac{(\beta+2\beta^2)S_k^{ii}u_i^2}{u^2}\\
 &\geq \varepsilon S_k^{ii}u_{ii}^2-2\dsum_{i=1}^\mu S_k^{ii}(\varepsilon u_iu_{ii})^2-\dfrac{1+1/\delta}{\beta}\dsum_{i>\mu}S_k^{ii}(\varepsilon u_iu_{ii})^2-C.\nonumber
\end{align}
We take $\varepsilon$ sufficiently small, such that
$$\varepsilon>\max\{4,\dfrac{2+2/\delta}{\beta}\}\varepsilon^2\displaystyle\max_\Omega|Du|^2.$$
Then we have
\begin{equation}\label{3.25}
\dfrac{1}{2}\varepsilon S_k^{ii}u_{ii}^2-2\dsum_{i=1}^\mu
S_k^{ii}(\varepsilon
u_iu_{ii})^2-\dfrac{1+1/\delta}{\beta}\dsum_{i>\mu}S_k^{ii}(\varepsilon
u_iu_{ii})^2>0.
\end{equation}
Using \eqref{3.24}, \eqref{3.25} and $\la_1S_k^{11}\geq\theta S_k$,
we obtain
\begin{align*}
-\dfrac{C}{u}+\dsum_{i=1}^\mu\dfrac{(\beta+2\beta^2)S_k^{ii}u_i^2}{u^2}
\geq&\dfrac{\varepsilon}{3} S_k^{ii}u_{ii}^2\geq
\dfrac{\varepsilon}{6} \dsum_{i=1}^\mu
S_k^{ii}u_{ii}^2+\dfrac{\varepsilon \theta f}{6} u_{11}.
\end{align*}
Hence, we get \eqref{1.4}. \vskip 2mm

\par
{\bf Case B:} $\lambda_n<-\dfrac{\delta\lambda_1}{3}$, then
$\la_n^2>\dfrac{\delta^2\lambda_1^2}{9}$. Using \eqref{3.19} and
\eqref{3.23}, we obtain
\begin{align}\label{3.26}
-\dfrac{C}{u}+\dfrac{(\beta+2\beta^2)S_k^{ii}u_i^2}{u^2}
\geq&\varepsilon S_k^{ii}u_{ii}^2-2 S_k^{ii}(\varepsilon
u_iu_{ii})^2-C.
\end{align}
We choose
$\varepsilon>4\varepsilon^2\displaystyle\max_\Omega|Du|^2$. It then
follows that
\begin{equation*}
-\dfrac{C}{u}+\dfrac{(\beta+2\beta^2)S_k^{ii}u_i^2}{u^2}\geq\dfrac{\varepsilon}{3}S_k^{ii}u_{ii}^2.
\end{equation*}
Hence, we complete the proof of Theorem \ref{th1.1}.

\section{Pogorelov type $C^2$ estimates for $k$-convex solutions}
\par
In this section, we will prove Theorem \ref{th1.2}. Similar to the
method in \cite{LRW1}, we consider the following test function.
\begin{equation*}
\phi=m\log(-u)+\log P_m+\dfrac{mN}{2}|Du|^2,
\end{equation*}
where
\begin{equation*}
P_m=\dsum_j\kappa_j^m, \kappa_j=\lambda_j+K_0,
\end{equation*}
$m$ and $N$ are some sufficiently large undetermined constants. The
$\lambda_1,\lambda_2,\cdots,\lambda_n$ are eigenvalues of the
Hessian matrix $D^2u$. Let $K_0=n\Big(\dfrac{\sup_\Omega
f}{\alpha}\Big)^{\frac{1}{k-1}}$. By \eqref{2.4},
$\kappa_1,\kappa_2,\cdots,\kappa_n$ are nonnegative. Suppose that
function $\phi$ achieves its maximum value in $\Omega$ at some point
$x_0$. By rotating the coordinates, we assume that $(u_{ij})$ is a
diagonal matrix at $x_0$, and
$\kappa_1\geq\kappa_2\geq\cdots\geq\kappa_n$.
\par
Differentiating test function twice and using Lemma \ref{lem2.2}, at
$x_0$, we have
\begin{equation}\label{4.1}
\dfrac{\dsum_j\kappa_j^{m-1}u_{jji}}{P_m}+Nu_iu_{ii}+\dfrac{u_i}{u}=0,
\end{equation}
and
\begin{align}\label{4.2}
0\geq&\dfrac{1}{P_m}\Big[\dsum_j\kappa_j^{m-1}u_{jjii}+(m-1)\dsum_j\kappa_j^{m-2}u_{jji}^2+\dsum_{p\neq
q}\dfrac{\kappa_p^{m-1}-\kappa_q^{m-1}}{\kappa_p-\kappa_q}u_{pqi}^2\Big]\\
&-\dfrac{m}{P_m^2}\Big(\dsum_j\kappa_j^{m-1}u_{jji}\Big)^2+\dsum_sNu_su_{sii}+Nu_{ii}^2+\dfrac{u_{ii}}{u}-\dfrac{u_i^2}{u^2}.\nonumber
\end{align}
\par
At $x_0$, differentiating the equation \eqref{1.1} twice, we have
\begin{equation}\label{4.3}
S_k^{ii}u_{iij}=f_j+f_uu_j+f_{p_j}u_{jj}
\end{equation}
and
\begin{equation}\label{4.4}
S_k^{ii}u_{iijj}+S_k^{pq,rs}u_{pqj}u_{rsj}\geq-C-Cu_{11}^2+\dsum_sf_{p_s}u_{sjj}.
\end{equation}
Here, $C$ is a constant depending on $f$, the diameter of the domain
$\Omega$, $\displaystyle\sup_\Omega |u|$ and
$\displaystyle\sup_\Omega |Du|$.  Contacting $S_k^{ii}$ in both
sides of \eqref{4.2}, using \eqref{4.4}, we have
\begin{align}\label{4.5}
0\geq& \dfrac{1}{P_m}\Big[\dsum_j\kappa_j^{m-1}\Big(-C-Cu_{11}^2+\dsum_sf_{p_s}u_{sjj}-K(S_k)_j^2+K(S_k)_j^2\\
       &-S_k^{pq,rs}u_{pqj}u_{rsj}\Big)+(m-1)S_k^{ii}\dsum_j\kappa_j^{m-2}u_{jji}^2+S_k^{ii}\dsum_{p\neq
q}\dfrac{\kappa_p^{m-1}-\kappa_q^{m-1}}{\kappa_p-\kappa_q}u_{pqi}^2\Big]\nonumber\\
       &-\dfrac{mS_k^{ii}}{P_m^2}\Big(\dsum_j\kappa_j^{m-1}u_{jji}\Big)^2+\dsum_sNu_sS_k^{ii}u_{sii}+NS_k^{ii}u_{ii}^2+\dfrac{S_k^{ii}u_{ii}}{u}-\dfrac{S_k^{ii}u_i^2}{u^2}.\nonumber
\end{align}
Using \eqref{4.1} and \eqref{4.3}, we have
\begin{align}\label{4.6}
\dfrac{1}{P_m}\dsum_j\dsum_s\kappa_j^{m-1}f_{p_s}u_{sjj}+\dsum_sNu_sS_k^{ii}u_{sii}\geq-\dsum_sf_{p_s}\dfrac{u_s}{u}-C.
\end{align}
By Lemma \ref{lem2.2},
\begin{align}\label{4.7}
 -S_k^{pq,rs}u_{pqj}u_{rsj}=&-S_k^{pp,qq}u_{ppj}u_{qqj}+\dsum_{p\neq
q}S_k^{pp,qq}u_{pqj}^2\\
\geq&-S_k^{pp,qq}u_{ppj}u_{qqj}+2\sum_{j\neq
i}S_k^{jj,ii}u_{jji}^2.\nonumber
\end{align}
For the given index $1\leq i,j\leq n$, we denote
\begin{align*}
   &A_i=\dfrac{\kappa_i^{m-1}}{P_m}\Big[K(S_k)_i^2-\dsum_{p,q}S_k^{pp,qq}u_{ppi}u_{qqi}\Big], ~~B_i=\dfrac{2\kappa_j^{m-1}}{P_m}\dsum_jS_k^{jj,ii}u_{jji}^2,\\
   &C_i=\dfrac{m-1}{P_m}S_k^{ii}\dsum_j\kappa_j^{m-2}u_{jji}^2, \qquad \qquad\qquad D_i=\dfrac{2S_k^{jj}}{P_m}\dsum_{j\neq i}\dfrac{\kappa_j^{m-1}-\kappa_i^{m-1}}{\kappa_j-\kappa_i}u_{jji}^2,\\
   &E_i=\dfrac{mS_k^{ii}}{P_m^2}\Big(\dsum_j\kappa_j^{m-1}u_{jji}\Big)^2.
\end{align*}
Using \eqref{4.6}, \eqref{4.7} and the above definitions,
\eqref{4.5} becomes
\begin{align}\label{4.8}
0\geq&\dsum_i(A_i+B_i+C_i+D_i-E_i)+NS_k^{ii}u_{ii}^2+\dfrac{S_k^{ii}u_{ii}}{u}-\dfrac{S_k^{ii}u_i^2}{u^2}\\
&-\dsum_sf_{p_s}\dfrac{u_s}{u}-C(K)u_{11}.\nonumber
\end{align}

\par
Let's deal with the third derivatives. We divide two cases $i\neq 1$
and $i=1$. The proof is the same as in \cite{LRW1}.
\begin{lemma}\label{lem4.1}
For any $i\neq1$, we have
$$A_i+B_i+C_i+D_i-(1+\dfrac{1}{m})E_i\geq0,$$
for sufficiently large $m$ and $\lambda_1$.
\end{lemma}
\begin{proof}
First, let $l=1$ in the formula \eqref{2.2} of Lemma \ref{lem2.1},
then we have $A_i\geq0$ for sufficiently large $K$.
\par
Next, it is the same as (3.13) in \cite{LRW1}. Using
\begin{align*}
\kappa_jS_k^{jj,ii}+S_k^{jj}&=(\lambda_j+K_0)S_k^{jj,ii}+S_k^{jj}\\
&=K_0S_k^{jj,ii}+S_k^{ii}-S_{k-1}(\lambda|ij)+\lambda_iS_{k-2}(\lambda|ij)+S_{k-1}(\lambda|ij)\\
&=(\lambda_i+K_0)S_k^{ii,jj}+S_k^{ii} \geq  S_k^{ii},
\end{align*}
and Cauchy-Schwarz inequalities  we obtain
\begin{align}\label{4.9}
    P_m^2&\Big[B_i+C_i+D_i-(1+\dfrac{1}{m})E_i\Big]\\
      &\geq\dsum_{j\neq i}\Big[(m+1)\kappa_i^m\kappa_j^{m-2}S_k^{ii}+2\kappa_1^mS_k^{jj}\dsum_{l=0}^{m-3}\kappa_i^{m-2-l}\kappa_j^l\Big]u_{jji}^2\nonumber\\
    &+\Big[(m-1)(P_m-\kappa_i^m)-2\kappa_i^m\Big]S_k^{ii}\kappa_i^{m-2}u_{iii}^2\nonumber\\
    &-2(m+1)S_k^{ii}\kappa_i^{m-1}u_{iii}\dsum_{j\neq
    i}\kappa_j^{m-1}u_{jji}.\nonumber
\end{align}
By \eqref{2.5},  we have
\begin{equation*}
2\kappa_1^mS_k^{jj}\dsum_{l=0}^{m-3}\kappa_i^{m-2-l}\kappa_j^l\geq 7
S_k^{ii}\kappa_j^{m-2}\kappa_i^m
\end{equation*}
for $m\geq\dmax\{k+2,10\}$. Using the above formula and \eqref{4.9},
\begin{align*}
     P_m^2\Big[&B_i+C_i+D_i-(1+\dfrac{1}{m})E_i\Big]\nonumber\\
     \geq&\dsum_{j\neq i}(m+8)\kappa_i^m\kappa_j^{m-2}S_k^{ii}u_{jji}^2+[(m-1)(P_m-\kappa_i^m)-2\kappa_i^m]S_k^{ii}\kappa_i^{m-2}u_{iii}^2\\
    &-2(m+1)S_k^{ii}\kappa_i^{m-1}u_{iii}\dsum_{j\neq
    i}\kappa_j^{m-1}u_{jji}\\
    \geq&(m+8)\kappa_i^m\kappa_1^{m-2}S_k^{ii}u_{11i}^2+[(m-1)\kappa_1^m-2\kappa_i^m]S_k^{ii}\kappa_i^{m-2}u_{iii}^2\\
    &-2(m+1)S_k^{ii}\kappa_i^{m-1}u_{iii}\kappa_1^{m-1}u_{11i}\\
    \geq&(m+8)\kappa_i^m\kappa_1^{m-2}S_k^{ii}u_{11i}^2+(m-3)\kappa_1^mS_k^{ii}\kappa_i^{m-2}u_{iii}^2\\
    &-2(m+1)S_k^{ii}\kappa_i^{m-1}u_{iii}\kappa_1^{m-1}u_{11i}\\
    \geq&0.
\end{align*}
In the second inequality we have used, for $m\geq10$
$$\dsum_{j\neq i,1}(m+8)\kappa_i^m\kappa_j^{m-2}S_k^{ii}u_{jji}^2+(m-1)\dsum_{j\neq i,1}\kappa_j^m\kappa_i^{m-2}S_k^{ii}u_{iii}^2$$
$$-2(m+1)S_k^{ii}\kappa_i^{m-1}u_{iii}\dsum_{j\neq
    i,1}\kappa_j^{m-1}u_{jji}\geq0.$$
In the forth inequality we have used
$$(m+8)(m-3)\geq(m+1)^2.$$
Take $m=\max\{k+2,10\}$. Then we obtain Lemma \ref{lem4.1}.
\end{proof}

\begin{lemma}\label{lem4.2}
For $\mu=1,\cdots,k-2$, if there exists some positive constant
$\delta\leq1$, such that $\dfrac{\lambda_\mu}{\lambda_1}\geq\delta$,
then there exist two sufficiently small positive constants
$\eta,\delta'$ depending on $\delta$, such that, if
$\dfrac{\lambda_{\mu+1}}{\lambda_1}\leq\delta'$, we have
$$A_1+B_1+C_1+D_1-(1+\dfrac{\eta}{m})E_1\geq0.$$
\end{lemma}
\begin{proof}
First, it is the same as (3.19) in \cite{LRW1}. For $m\geq5$, we
have
\begin{align}\label{4.10}
    P_m^2\Big[&B_1+C_1+D_1-(1+\dfrac{\eta}{m})E_1\Big]\\
   \geq&-(1+\eta)S_k^{11}\kappa_1^{2m-2}u_{111}^2+2P_m\kappa_1^{m-2}\dsum_{j\neq1}S_k^{jj}u_{jj1}^2.\nonumber
\end{align}
\par
Next, we will replace $-(1+\eta)S_k^{11}\kappa_1^{2m-2}u_{111}^2$
from $A_1$. We divide two cases $\mu=1$ and $2\leq\mu\leq k-2$.
\par
For $\mu=1$, since $S_1=\sigma_1+\alpha$ and $S_1^{aa,bb}=0$, by
Lemma \ref{lem2.1}, for sufficiently large $K$, we have
\begin{align}\label{4.11}
    A_1&=\dfrac{\kappa_1^{m-1}}{P_m}\Big[K(S_k)_1^2-\dsum_{p,q}S_k^{pp,qq}u_{pp1}u_{qq1}\Big]\\
   &\geq\dfrac{\kappa_1^{m-1}}{P_m}S_k(1+\dfrac{\vartheta}{2})\dfrac{(S_1)_1^2}{S_1^2}.\nonumber
\end{align}
Using $S_1^{aa}=1$, we have
\begin{align}\label{4.12}
    (1+\dfrac{\vartheta}{2})(S_1)_1^2=&(1+\dfrac{\vartheta}{2})\Big(\dsum_au_{aa1}\Big)^2\\
    =&(1+\dfrac{\vartheta}{2})u_{111}^2+2(1+\dfrac{\vartheta}{2})\dsum_{a\neq1}u_{aa1}u_{111}+(1+\dfrac{\vartheta}{2})\Big(\dsum_{a\neq1}u_{aa1}\Big)^2\nonumber\\
   \geq&(1+\dfrac{\vartheta}{4})u_{111}^2-C_\vartheta\dsum_{a\neq1}u_{aa1}^2.\nonumber
\end{align}
Combining \eqref{4.11}, \eqref{4.12} and \eqref{2.51}, we obtain
\begin{align}\label{4.13}
P_m^2A_1&\geq\dfrac{P_m\kappa_1^{m-1}S_k}{S_1^2}(1+\dfrac{\vartheta}{4})u_{111}^2-\dfrac{P_m\kappa_1^{m-1}C_\vartheta}{S_1^2}\dsum_{a\neq1}u_{aa1}^2\\
    &\geq\dfrac{P_m\kappa_1^{m-2}S_k^{11}(1-\varepsilon_0)}{(1+\dsum_{j\neq1}\dfrac{\lambda_j}{\lambda_1}+\dfrac{\alpha}{\lambda_1})^2}(1+\dfrac{\vartheta}{4})u_{111}^2-\dfrac{P_m\kappa_1^{m-1}C_\vartheta}{S_1^2}\dsum_{a\neq1}u_{aa1}^2\nonumber\\
   &\geq(1+\eta)\kappa_1^{2m-2}S_k^{11}u_{111}^2-P_m\kappa_1^{m-3}C_\vartheta\dsum_{a\neq1}u_{aa1}^2.\nonumber
\end{align}
The last inequality comes from choosing positive constants
$\delta',\varepsilon_0,\eta$ such that
$$\dfrac{\lambda_j}{\lambda_1}\leq \delta', \quad (1+\dfrac{\vartheta}{4})(1-\varepsilon_0)\geq(1+\eta)(1+n\delta')^2.$$

\par
For $2\leq\mu\leq k-2$, by Lemma \ref{lem2.1}, we have
\begin{align}\label{4.14}
   A_1\geq&\dfrac{\kappa_1^{m-1}}{P_m}\Big[S_k\dfrac{(S_\mu)_1^2}{S_\mu^2}-\dfrac{S_k}{S_\mu}S_\mu^{pp,qq}u_{pp1}u_{qq1}\Big]\\
    =&\dfrac{\kappa_1^{m-1}S_k}{P_mS_\mu^2}\Big[\dsum_a(S_\mu^{aa}u_{aa1})^2+\dsum_{a\neq b}(S_\mu^{aa}S_\mu^{bb}-S_\mu S_\mu^{aa,bb})u_{aa1}u_{bb1}\Big].\nonumber
\end{align}
Next, we split $\dsum_{a\neq b}(S_\mu^{aa}S_\mu^{bb}-S_\mu
S_\mu^{aa,bb})u_{aa1}u_{bb1}$ into three terms to deal with:
$$
\dsum_{a\neq b}=\dsum_{a\neq b;a,b\leq\mu}+2\dsum_{a\leq
\mu,b>\mu}+\dsum_{a\neq b;a,b>\mu}:=T_1+T_2+T_3
$$
\par
First, let's deal with $T_1$. By Lemma \ref{lem2.5}, for $a\neq b$,
we have
\begin{align}\label{4.15}
S_\mu^{aa}S_\mu^{bb}-S_\mu
S_\mu^{aa,bb}=S_{\mu-1}^2(\lambda|ab)-S_\mu(\lambda|ab)S_{\mu-2}(\lambda|ab)\geq
0.
\end{align}
Note that $\lambda\in\Gamma_k, \mu\leq k-2$, which implies
$S_\mu(\lambda|ab)>0$. So \eqref{4.15} implies
\begin{align}\label{4.16}
T_1=&\dsum_{a\neq b;a,b\leq\mu}(S_\mu^{aa}S_\mu^{bb}-S_\mu S_\mu^{aa,bb})u_{aa1}u_{bb1}\\
\geq&-\dsum_{a\neq
b;a,b\leq\mu}\Big[S_{\mu-1}^2(\lambda|ab)-S_\mu(\lambda|ab)S_{\mu-2}(\lambda|ab)\Big]u_{aa1}^2\nonumber\\
\geq&-\dsum_{a\neq
b;a,b\leq\mu}S_{\mu-1}^2(\lambda|ab)u_{aa1}^2.\nonumber
\end{align}
By Lemma \ref{lem2.4a}, for any $a,b \leq\mu$,
\begin{equation}\label{4.17}
S_\mu^{aa}=\sigma_{\mu-1}(\la|a)+\al\sigma_{\mu-2}(\la|a)\geq\dfrac{\lambda_1\cdots\lambda_\mu}{\lambda_a},
\end{equation}
\begin{equation}\label{4.18}
S_{\mu-1}(\lambda|ab)\leq
C(\dfrac{\lambda_1\cdots\lambda_{\mu+1}}{\lambda_a\lambda_b}+\alpha\dfrac{\lambda_1\cdots\lambda_\mu}{\lambda_a\lambda_b})\leq
C(\dfrac{\alpha+\lambda_{\mu+1}}{\lambda_b})S_\mu^{aa}.
\end{equation}
Using \eqref{4.16}, \eqref{4.17}, \eqref{4.18}, we obtain
\begin{align}\label{4.19}
T_1\geq&-\dsum_{a\neq b;a,b\leq\mu}C_1\dfrac{(\alpha+\lambda_{\mu+1})^2}{\lambda_b^2}(S_\mu^{aa}u_{aa1})^2\\
   \geq&-\dfrac{C_1}{\delta^2}\dfrac{(\al+\delta'\lambda_1)^2}{\lambda_1^2}\dsum_{a\leq\mu}(S_\mu^{aa}u_{aa1})^2\nonumber\\
   \geq&-\varepsilon \dsum_{a\leq\mu}(S_\mu^{aa}u_{aa1})^2.\nonumber
\end{align}
Here, we choose positive constants $\delta',\varepsilon$
sufficiently small and $\lambda_1$ sufficiently large, satisfying
\begin{equation*}
\dfrac{4C_1\delta'^2}{\delta^2}<\varepsilon,\quad
\al<\delta'\lambda_1.
\end{equation*}
Next, let's deal with $T_2$ and $T_3$. Using Cauchy-Schwarz
inequalities we get
\begin{align}
T_2\geq&-2\dsum_{a\leq\mu;b>\mu}S_\mu^{aa}S_\mu^{bb}|u_{aa1}u_{bb1}|\label{4.20}\\
    \geq&-\varepsilon\dsum_{a\leq\mu;b>\mu}(S_\mu^{aa}u_{aa1})^2-\dfrac{1}{\varepsilon}\dsum_{a\leq\mu;b>\mu}(S_\mu^{bb}u_{bb1})^2,\nonumber\\
    T_3\geq& -\dsum_{a\neq b;a,b>\mu}S_\mu^{aa}S_\mu^{bb}|u_{aa1}u_{bb1}|\geq-\dsum_{a\neq
b;a,b>\mu}(S_\mu^{aa}u_{aa1})^2.\label{4.21}
\end{align}

Hence, combining \eqref{4.14}, \eqref{4.19}, \eqref{4.20} and
\eqref{4.21}, we obtain
\begin{equation}\label{4.22}
A_1\geq\dfrac{\kappa_1^{m-1}S_k}{P_mS_\mu^2}\Big[(1-2\varepsilon)\dsum_{a\leq\mu}(S_\mu^{aa}u_{aa1})^2-C_\varepsilon\dsum_{a>\mu}(S_\mu^{aa}u_{aa1})^2\Big].
\end{equation}

For $a>\mu$, we have
\begin{equation}\label{4.23}
S_\mu^{aa}\leq C_2\lambda_1\cdots\lambda_{\mu-1},\quad
S_\mu\geq\lambda_1\cdots\lambda_\mu.
\end{equation}
So
\begin{equation}\label{4.24}
-\dfrac{1}{S_\mu^2}\dsum_{a>\mu}(S_\mu^{aa}u_{aa1})^2\geq-\dfrac{C_2^2}{\lambda_{\mu}^2}\dsum_{a>\mu}u_{aa1}^2
\geq -\dfrac{C_3}{\kappa_1^2\delta^2}\dsum_{a>\mu}u_{aa1}^2.
\end{equation}

For $a\leq\mu$, we have
\begin{equation}\label{4.25}
S_\mu(\lambda|a)\leq
C_4(\dfrac{\lambda_1\cdots\lambda_{\mu+1}}{\lambda_a}+\alpha\dfrac{\lambda_1\cdots\lambda_\mu}{\lambda_a}).
\end{equation}
Using \eqref{4.23}, \eqref{4.25} and \eqref{2.51}, we have
\begin{align}\label{4.26}
\dfrac{S_k}{S_\mu^2}\dsum_{a\leq\mu}(S_\mu^{aa}u_{aa1})^2\geq&\dfrac{\lambda_1S_k^{11}(1-\varepsilon_0)}{S_\mu^2}\dsum_{a\leq\mu}(S_\mu^{aa}u_{aa1})^2\\
\geq&\dfrac{S_k^{11}(1-\varepsilon_0)}{\lambda_1}\dsum_{a\leq\mu}\Big(\dfrac{\lambda_aS_\mu^{aa}}{S_\mu}\Big)^2u_{aa1}^2\nonumber\\
    \geq&\dfrac{S_k^{11}(1-\varepsilon_0)}{\kappa_1}\dsum_{a\leq\mu}\Big(1-\dfrac{S_{\mu}(\la|a)}{S_\mu}\Big)^2u_{aa1}^2\nonumber\\
    \geq&\dfrac{S_k^{11}(1-\varepsilon_0)}{\kappa_1}\Big(1-\dfrac{C_5(\alpha+\delta'\lambda_1)}{\delta\lambda_1}\Big)^2\dsum_{a\leq\mu}u_{aa1}^2\nonumber\\
    \geq&\dfrac{S_k^{11}(1-\varepsilon_0)}{\kappa_1}\Big(1-\dfrac{C_6\delta'}{\delta}\Big)^2\dsum_{a\leq\mu}u_{aa1}^2.\nonumber
\end{align}

Combining \eqref{4.22}, \eqref{4.24}, \eqref{4.26} and
$P_m>(1+\delta^m)\kappa_1^m$, we obtain
\begin{align}\label{4.27}
P_m^2A_1\geq&(1+\delta^m)\kappa_1^{2m-2}S_k^{11}(1-\varepsilon_0)(1-2\varepsilon)\Big(1-\dfrac{C_6\delta'}{\delta}\Big)^2\dsum_{a\leq\mu}u_{aa1}^2\\
&-\dfrac{P_m\kappa_1^{m-3}C_{\varepsilon}'S_k}{\delta^2}\dsum_{a>\mu}u_{aa1}^2\nonumber\\
\geq&(1+\eta)\kappa_1^{2m-2}S_k^{11}\dsum_{a\leq\mu}u_{aa1}^2
 -\dfrac{P_m\kappa_1^{m-3}C_{\varepsilon}'S_k}{\delta^2}\dsum_{a>\mu}u_{aa1}^2.\nonumber
\end{align}
Here, the last inequality comes from choosing positive constants
$\varepsilon_0,\varepsilon, \delta',\eta$ sufficiently small such
that
\begin{equation*}
\dfrac{C_6\delta'}{\delta}\leq 2\varepsilon,\quad
(1-\varepsilon_0)(1-2\varepsilon)^3(1+\delta^m)\geq 1+\eta.
\end{equation*}
\par
Using \eqref{4.10}, \eqref{4.13} or \eqref{4.27}, we have
\begin{align}\label{4.28}
P_m^2\Big[&A_1+B_1+C_1+D_1-(1+\dfrac{\eta}{m})E_1\Big]\\
\geq&2P_m\kappa_1^{m-2}\dsum_{j\neq1}S_k^{jj}u_{jj1}^2-\dfrac{P_m\kappa_1^{m-3}C_\varepsilon'S_k}{\delta^2}\dsum_{j>\mu}u_{jj1}^2.\nonumber
\end{align}
Now, for $k-1\geq j>\mu$, by \eqref{2.53}, we have
\begin{equation*}
\begin{aligned}
\kappa_1S_k^{jj}&=\dfrac{\kappa_1}{\la_j}\la_jS_k^{jj}\geq\dfrac{\kappa_1}{\la_j}\theta
S_k\geq\dfrac{\theta S_k}{\delta'}.
\end{aligned}
\end{equation*}
For $j\geq k$, we have
$$\kappa_1S_k^{jj}\geq\kappa_1S_k^{k-1,k-1}\geq\dfrac{\kappa_1}{\la_{k-1}}\theta S_k\geq\dfrac{\theta S_k}{\delta'}.$$
Hence, choose $\delta'$ small enough such that
$$\delta'<\dfrac{\theta \delta^2}{C_\varepsilon'},$$
then \eqref{4.28} is nonnegative. We complete the proof of the
lemma.
\end{proof}

\par
Hence, a direct corollary of Lemma \ref{lem4.2} is the following.
\begin{corollary}\label{cor4.1}
There exist two finite sequences of positive numbers
$\{\delta_i\}_{i=1}^{k-1}$ and $\{\eta_i\}_{i=1}^{k-1}$, such that,
if the following inequality holds for some index $1\leq r\leq k-2$,
$$\dfrac{\lambda_r}{\lambda_1}\geq\delta_r,\quad \dfrac{\lambda_{r+1}}{\lambda_1}\leq\delta_{r+1},$$
then, for sufficiently large $K$ and $\lambda_1$, we have
\begin{equation}
A_1+B_1+C_1+D_1-(1+\dfrac{\eta_r}{m})E_1\geq0.
\end{equation}
\end{corollary}

\par
Now, we can prove Theorem \ref{th1.2}.
\par
By Corollary \ref{cor4.1}, there exists some sequence
$\{\delta_i\}_{i=1}^{k-1}$. We consider the following two cases.
\par
Case 1: If $r=k-1$,$\lambda_{k-1}\geq\delta_{k-1}\lambda_1$,
obviously we have
$$f=\sigma_k+\alpha\sigma_{k-1}>\alpha\lambda_1\cdots\lambda_{k-1}\geq\delta_{k-1}^{k-2}\lambda_1^{k-1},$$
which implies $\lambda_1\leq C$. Hence, we complete the proof of
Theorem \ref{th1.2}.
\par
Case 2: There exists some index $1\leq r\leq k-2$, such that
$$\dfrac{\lambda_r}{\lambda_1}\geq\delta_r,\quad\dfrac{\lambda_{r+1}}{\lambda_1}\leq\delta_{r+1}.$$
Using Lemma \ref{lem4.1} and Corollary \ref{cor4.1}, we have
\begin{equation}\label{4.30}
\dsum_i(A_i+B_i+C_i+D_i)-E_1-(1+\dfrac{1}{m})\dsum_{i=2}^nE_i\geq0.
\end{equation}
Combining \eqref{4.8} with \eqref{4.30}, we have
\begin{align}\label{4.31}
0\geq&\dsum_{i=2}^n\dfrac{S_k^{ii}}{P_m^2}(\dsum_j\kappa_j^{m-1}u_{jji})^2+Nu_{ii}^2S_k^{ii}+\dfrac{S_k^{ii}u_{ii}}{u}-\dfrac{S_k^{ii}u_i^2}{u^2}-\dsum_sf_{p_s}\dfrac{u_s}{u}-Cu_{11}.
\end{align}
By \eqref{4.1}, for any fixed $i\geq2$, we have
\begin{align*}
-\dfrac{S_k^{ii}u_i^2}{u^2}=&-\dfrac{S_k^{ii}}{P_m^2}(\dsum_j\kappa_j^{m-1}u_{jji})^2+S_k^{ii}N^2u_i^2u_{ii}^2+\dfrac{2NS_k^{ii}u_i^2u_{ii}}{u}\nonumber\\
\geq&-\dfrac{S_k^{ii}}{P_m^2}(\dsum_j\kappa_j^{m-1}u_{jji})^2+\dfrac{2NC_0S_ku_i^2}{u}.
\end{align*}
In the above inequality, we have used \eqref{2.52}. Using the above
formula and (v) in section 2, (\ref{4.31}) becomes
\begin{align}
-\dfrac{C}{u}\geq&
Nu_{ii}^2S_k^{ii}-\dfrac{u_1^2S_k^{11}}{u^2}-Cu_{11}.
\end{align}
Since $\lambda\in\Gamma_k$, we know $u_{11}S_k^{11}\geq \theta S_k$
by \eqref{2.53}. Choose $N$ sufficiently large in the above formula,
such that
$$\dfrac{N}{2}u_{ii}^2S_k^{ii}-Cu_{11}>0,$$
we obtain
$$-\dfrac{C}{u}+\dfrac{CS_k^{11}}{u^2}\geq\dfrac{N}{4}S_k^{11}\lambda_1^2+\dfrac{N\theta S_k}{4}\lambda_1.$$
Hence, we complete the proof of Theorem \ref{th1.2}.

\section{A rigidity theorem for $k$ convex solutions}
In this section, we prove Theorem \ref{th1.3}. First, we have the
following lemma.
\begin{lemma}\label{lem5.1}
For the following Dirichlet problem of Sum Hessian equations,
\begin{equation}\label{5.1}
\left\{\begin{array}{lll}
 S_k(D^2u)=f (x),  &in &\Omega,\\
u=0,  & on&\pa{\Omega}.
\end{array} \right.
\end{equation}
Assume that $f>0$ is a smooth function in some domain $\Omega$, $u$
is a $k$ convex solution of the above problem.  Then there exists
some sufficiently large constant $\beta>0$, such that
\begin{equation}\label{5.2}
(-u)^\beta\Delta u\leq C.
\end{equation}
Here, the constants $C$ and $\beta$ only depend on $n, k$ and the
diameters of the domain $\Omega$.
\end{lemma}
\begin{proof}
Obviously, for sufficiently large $a$ and $b$, the function
$w=\dfrac{a}{2}|x|^2-b$ can control $u$ by comparison
principal~\cite{CNS3}, namely
$$w\leq u\leq 0.$$
Here $a$ and $b$ depend on the diameter of the domain $\Omega$.
Hence, in the following proof, the constant $\beta$, $C$ in
\eqref{5.2} can contain $\displaystyle\sup_\Omega |u|$.
\par
Since $u$ is $k$ convex solution, by (\ref{2.4}), there is some
constant $K_0>0$, such that $D^2u+K_0I\geq0$. We consider the test
function,
$$\phi=m\beta \log(-u)+\log P_m+\dfrac{m}{2}|x|^2,$$
where
\begin{equation*}
P_m=\dsum_j\kappa_j^m, \quad \kappa_j=\lambda_j+K_0.
\end{equation*}
Suppose $\phi$ achieves its maximum value at $x_0$ in $\Omega$. By
rotating the coordinate, we assume that $(u_{ij})$ is a diagonal
matrix at $x_0$, so $u_{ii}=\lambda_i$ and
$\kappa_1\geq\kappa_2\geq\cdots\geq\kappa_n$.
\par
Differentiating test function twice at $x_0$, we have
\begin{equation}\label{5.3}
\dfrac{\dsum_j\kappa_j^{m-1}u_{jji}}{P_m}+x_i+\dfrac{\beta u_i}{u}=0
\end{equation}
and
\begin{align}\label{5.4}
0\geq&\dfrac{1}{P_m}\Big[\dsum_j\kappa_j^{m-1}u_{jjii}+(m-1)\dsum_j\kappa_j^{m-2}u_{jji}^2+\dsum_{p\neq
q}\dfrac{\kappa_p^{m-1}-\kappa_q^{m-1}}{\kappa_p-\kappa_q}u_{pqi}^2\Big]\\
&-\dfrac{m}{P_m^2}\Big(\dsum_j\kappa_j^{m-1}u_{jji}\Big)^2+\dfrac{\beta
u_{ii}}{u}-\dfrac{\beta u_i^2}{u^2}+1.\nonumber
\end{align}
\par
At $x_0$, differentiating the equation \eqref{5.1} twice, we have
\begin{equation}\label{5.5}
S_k^{ii}u_{iij}=f_j
\end{equation}
and
\begin{equation}\label{5.6}
S_k^{ii}u_{iijj}+S_k^{pq,rs}u_{pqj}u_{rsj}=f_{jj}.
\end{equation}
Using  \eqref{5.3}, we have
\begin{equation}\label{5.7}
-\dfrac{\beta
u_i^2}{u^2}\geq-\dfrac{2}{\beta}\dfrac{\Big(\dsum_j\kappa_j^{m-1}u_{jji}\Big)^2}{P_m^2}-\dfrac{2x_i^2}{\beta}.
\end{equation}
Contacting $S_k^{ii}$ in both sides of \eqref{5.4}, using
\eqref{5.6} and \eqref{5.7}, we obtain
\begin{align}\label{5.8}
0\geq& \dfrac{1}{P_m}\Big[\dsum_j\kappa_j^{m-1}(f_{jj}-S_k^{pq,rs}u_{pqj}u_{rsj})+(m-1)S_k^{ii}\dsum_j\kappa_j^{m-2}u_{jji}^2\\
       &+S_k^{ii}\dsum_{p\neq q}\dfrac{\kappa_p^{m-1}-\kappa_q^{m-1}}{\kappa_p-\kappa_q}u_{pqi}^2\Big]
       -\dfrac{(m+\frac{2}{\beta})S_k^{ii}}{P_m^2}\Big(\dsum_j\kappa_j^{m-1}u_{jji}\Big)^2\nonumber\\
       &+\dfrac{\beta S_k^{ii}u_{ii}}{u}-\dfrac{2S_k^{ii}x_i^2}{\beta}+\dsum_iS_k^{ii}.\nonumber
\end{align}
Using the definitions of $A_i,B_i,C_i,D_i,E_i$ and \eqref{5.8}, we
have
\begin{align}\label{5.9}
0\geq&
\dsum_{i=1}^n\Big[A_i+B_i+C_i+D_i-(1+\dfrac{2}{m\beta})E_i\Big]
       +\dfrac{\beta
       S_k^{ii}u_{ii}}{u}\\&-\dfrac{2S_k^{ii}x_i^2}{\beta}+\dsum_iS_k^{ii}-\dfrac{C}{\kappa_1}.\nonumber
\end{align}

For sequence $\{\delta_i\}_{i=1}^{k-1}$ and $\{\eta_i\}_{i=1}^{k-1}$
appearing in Corollary \ref{cor4.1}, we choose $\beta$ sufficiently
large such that
$$\dfrac{2}{\beta}<\min\{\dfrac{1}{2},\eta_1,\cdots,\eta_{k-1}\}.$$
If $r=k-1$, $\lambda_{k-1}\geq\delta_{k-1}\lambda_1$, obviously we
have
$$f=\sigma_k+\alpha\sigma_{k-1}>\alpha\lambda_1\cdots\lambda_{k-1}\geq\delta_{k-1}^{k-2}\lambda_1^{k-1},$$
which implies $\lambda_1\leq C$. Thus, we obtain Lemma \ref{lem5.1}.
\par
If there exists some index $1\leq r\leq k-2$, such that
$$\dfrac{\lambda_r}{\lambda_1}\geq\delta_r,\quad \dfrac{\lambda_{r+1}}{\lambda_1}\leq\delta_{r+1},$$
then \eqref{5.9} becomes
\begin{equation}\label{5.10}
0\geq\dfrac{\beta
S_k^{ii}u_{ii}}{u}-\dfrac{2S_k^{ii}x_i^2}{\beta}+\dsum_iS_k^{ii}-\dfrac{C}{\kappa_1}
\end{equation}
by Lemma \ref{lem4.1} and Corollary \ref{cor4.1}.  We choose $\beta$
sufficiently large, such that
$$
-\dfrac{2S_k^{ii}x_i^2}{\beta}+\dfrac{1}{2}\dsum_iS_k^{ii}>0.
$$
By Newton-Maclaurin inequality, we have
$$\sigma_{k-1}\geq\sigma_1^{\frac{1}{k-1}}\sigma_k^{\frac{k-2}{k-1}},~~~\sigma_{k-2}\geq\sigma_1^{\frac{1}{k-2}}\sigma_{k-1}^{\frac{k-3}{k-2}},$$
which implies
$$\dsum_iS_k^{ii}=(n-k+1)\sigma_{k-1}+\alpha(n-k+2)\sigma_{k-2}\geq c_0\sigma_1^{\frac{1}{k-1}}.$$
By (v) in section 2,
$$\dfrac{S_k^{ii}u_{ii}}{u}=\dfrac{kS_k-\al\sigma_{k-1}}{u}>\dfrac{kS_k}{u}.$$
Substituting the above inequalities into \eqref{5.10}, we obtain
$$-\dfrac{\beta C}{u}\geq \dfrac{1}{2}\dsum_iS_k^{ii}-\dfrac{C}{\kappa_1}\geq \dfrac{c_0}{4}\sigma_1^{\frac{1}{k-1}}.$$
Thus, we obtain Lemma \ref{lem5.1}.
\end{proof}

\par
\textbf{Proof of Theorem 1.3:} The proof is classical as in
\cite{TW, LRW1}. For the convenience of the readers, the complete
proof is given here. Suppose $u$ is an entire solution of the
equation \eqref{1.7}. For arbitrary positive constant $R>1$, we
consider the set
$$\Omega_R=\{y\in R^n;u(Ry)\leq R^2\}.$$
Let
$$v(y)=\dfrac{u(Ry)-R^2}{R^2}.$$
We consider the following Dirichlet problem
\begin{equation}
\left\{\begin{array}{lll}
 {{\sigma }_{k}}( D^2v)+\alpha{{\sigma
}_{k-1}}( D^2v)=1,  &in &\Omega_R,\\
v=0,  & on&\pa{\Omega_R}.
\end{array} \right.
\end{equation}
By Lemma \ref{lem5.1}, we have the following estimates
\begin{equation}\label{5.12}
(-v)^\beta\Delta v\leq C.
\end{equation}
Here the constants $C$ and $\beta$ depend on $k$ and the diameter of
the domain $\Omega_R$. Now using the quadratic growth condition
appears in Theorem \ref{th1.3}, we have
$$c|Ry|^2-b\leq u(Ry)\leq R^2.$$
Namely
$$|y|^2\leq\dfrac{1+b}{c}.$$
Hence, $\Omega_R$ is bounded. Thus, the constant $C$ and $\beta$
become two absolutely constants. Consider the domain
$$\Omega_R'=\{y;u(Ry)\leq \dfrac{R^2}{2}\}\subset\Omega_R.$$
In $\Omega_R'$, we have
$$v(y)\leq-\dfrac{1}{2}.$$
Hence, \eqref{5.12} implies that in $\Omega_R'$, we have
\begin{equation*}
\Delta v\leq 2^\beta C.
\end{equation*}
Note that
$$\nabla_y^2v=\nabla_x^2u.$$
Thus, using the above two formulas, in
$\Omega_R'=\{x;u(x)\leq\dfrac{R^2}{2}\}$, we have
\begin{equation*}
\Delta u\leq C,
\end{equation*}
where $C$ is an absolutely constant. Since $R$ is arbitrary, we have
the above inequality in whole $R^n$. Using Evan-Krylov
theory~\cite{E}, we have
$$|D^2u|_{C^\alpha(B_R)}\leq C\dfrac{|D^2u|_{C^0(B_R)}}{R^\alpha}\leq\dfrac{C}{R^\alpha}.$$
Hence, when $R\rightarrow +\infty$, we obtain Theorem \ref{1.3}.

\section*{Acknowledgments:}

The last author wish to thank Professor Pengfei Guan for his
valuable suggestions and comments. We wish to thank the referees for
useful comments and suggestions.

\end{document}